\documentclass{amsart}

\newtheorem{theorem}{Theorem}[section]
\newtheorem{lemma}[theorem]{Lemma}
\newtheorem{proposition}[theorem]{Proposition}
\newtheorem{corollary}[theorem]{Corollary}

\theoremstyle{definition}

\newtheorem{example}[theorem]{Example}

\theoremstyle{remark}
\newtheorem{remark}[theorem]{Remark}

\numberwithin{equation}{section}

\begin{document}

\title[On the structure of complete k\"ahlerian manifolds]{On the structure of complete k\"ahlerian manifolds furnished with closed conformal vector fields}

\author{A. Caminha$^1$}
\address{$^1$Departamento de Matem\'atica, Universidade Federal do Cear\'a, Fortaleza,
Cear\'a, Brazil. 60455-760}
\email{caminha@mat.ufc.br}

\thanks{The author wishes to thank professor Jorge V. Pereira for some helpful conversations.}



\begin{abstract}
We show that if a connected compact k\"ahlerian surface $M$ with nonpositive gaussian curvature is furnished with a closed conformal vector field $\xi$ whose singular points are isolated, then $M$ is isometric to a flat torus and $\xi$ is parallel. We also consider the case of a connected complete k\"ahlerian manifod $M$ of complex dimension $n>1$ and furnished with a nontrivial closed conformal vector field $\xi$. In this case, it is well known that the singularities of $\xi$ are automatically isolated and the nontrivial leaves of the distribution generated by $\xi$ and $J\xi$ are totally geodesic in $M$. Assuming that one such leaf is compact, has torsion normal holonomy group and that the holomorphic sectional curvature of $M$ along it is nonpositive, we show that $\xi$ is parallel and $M$ is foliated by a family of totally geodesic isometric tori and also by a family of totally geodesic isometric complete k\"ahlerian manifolds of complex dimension $n-1$. In particular, the the universal covering of $M$ is isometric to a riemannian product having $\mathbb R^2$ as a factor. We also present a generic example showing that one cannot get rid of the hypothesis on the nonpositivity of the holomorphic sectional curvature along at least one such leaf. 
\end{abstract}

\maketitle

\section{Introduction}

The geometry of submanifolds of riemannian spaces in which either the submanifold or the ambient space is furnished with a closed conformal vector field has been the object of intense research in recent years. In this sense, a small collection of interesting works is the one composed by the papers \cite{Alias:12}, \cite{Caminha:09}, \cite{Lima:13}, \cite{Montiel:99}, \cite{Pitis:05} and \cite{Urbano:98} and the references therein.

In what concerns k\"ahlerian manifolds, a canonical class of examples having closed conformal vector fields is that of the warped products $M^n=I\times_fN^{2n-1}$ ($n$ standing for the {\em complex} dimension of $M$), where $N$ is a $(2n-1)-$dimensional sasakian manifold; in this case, $\xi=t\partial_t$ is closed conformal and $\psi=1$ is its conformal factor. Another one is that of the riemannian products $M^n=N^{n-1}\times\mathbb T$, where $N$ is a k\"ahlerian manifold of complex dimension $n-1$ and $\mathbb T$ is a flat torus of complex dimension $1$; in this case, the conformal vector field $\xi$ is the lift, to $M$, of a parallel vector field in $\mathbb T$, and $\psi=0$ is its conformal factor. 

In both of these classes of examples, the holomorphic sectional curvature of $M$ in the direction of $\xi$ vanishes identically and, if $J$ stands for the quasi-complex structure of $M$, then the leaves of the distribution generated by $\xi$ and $J\xi$ are totally geodesic in $M$. However, in the first class such leaves are noncompact, whereas in the second class they are compact.

Moreover, as we point out at the end of the paper, if $M^n$ is a compact complex symmetric space of complex dimension $n$, positive scalar curvature, vanishing first De Rham cohomology group and not isometric to $\mathbb S^{2n}$, then $M$ does not possess a nontrivial closed conformal vector field. In particular, connected compact k\"ahlerian manifols of positive constant holomorphic sectional curvature do not possess closed conformal vector fields. Thus, it seems that the presence of positive holomorphic sectional curvature tends to avoid the existence of nontrivial closed conformal vector fields.

The purpose of this paper is to show that, under a reasonable set of conditions on $\xi$, the second class of examples presented in the second paragraph is essentially the only one. More precisely, we first show that if a connected compact k\"ahlerian surface $M$, with nonpositive gaussian curvature, is furnished with a closed conformal vector field $\xi$ whose singular points are isolated, then $M$ is isometric to a flat torus and $\xi$ is parallel. We then consider the case of a connected complete k\"ahlerian manifod $M^n$, of complex dimension $n>1$ and furnished with a nontrivial closed conformal vector field $\xi$. In this case, it is a well known fact that the singularities of $\xi$ are automatically isolated and the nontrivial leaves of the distribution generated by $\xi$ and $J\xi$ are totally geodesic in $M$. Assuming that one such leaf is compact, has torsion normal holonomy group and that the holomorphic sectional curvature of $M$ along it is nonpositive, we show that $\xi$ is parallel and $M$ is foliated by a family of totally geodesic isometric tori, and also by a family of totally geodesic, isometric complete k\"ahlerian manifolds of complex dimension $n-1$. In particular, the the universal covering of $M$ is isometric to a riemannian product having $\mathbb R^2$ as a factor.

\section{Deforming k\"ahlerian metrics}

In the sequel, we let $(M^n,J,g)$ be a $n-$dimensional hermitian manifold, where $n$ stands for its complex dimension (hence, $M$ has real dimension $2n$). We also let $\omega\in\Omega^2(M)$ denote the corresponding k\"ahlerian form,
so that $\omega(X,Y)=\langle JX,Y\rangle$ for all $X,Y\in\mathfrak X(M)$. It is a standard fact 
that $M$ is a k\"ahlerian manifold if and only if $J$ is parallel with respect to the Levi-Civita connection $\nabla$ of $g$, i.e., if and only if
\begin{equation}\label{eq:parallelism of J in Kahler manifolds}
\nabla J=0.
\end{equation}

Whenever convenient, we write $g=\langle\cdot,\cdot\rangle$ and let $|\cdot|$ denote the corresponding norm. Also, for $X\in\mathfrak X(M)$, we let $\theta_X$ denote the $1-$form metrically dual to $X$, i.e., such that $\theta_X(Y)=\langle X,Y\rangle$ for $Y\in\mathfrak X(M)$; we also let $\theta_X^2$ denote the symmetrization of $\theta_X\otimes\theta_X$, i.e., the covariant symmetric $2-$tensor field on $M$ such that
$$\theta_X^2(Y,Z)=\theta_X(Y)\theta_X(Z)$$
for $Y,Z\in\mathfrak X(M)$.

The following result shows how to construct, out of $g$, a new k\"ahlerian metric on $(M,J)$. To this end, we recall that a vector field $\xi$ on $(M,g,\nabla)$ is said to be {\em closed conformal} if there exists a smooth function $\psi$ on $M$ (the {\em conformal factor} of $\xi$) such that
$$\nabla_X\xi=\psi X$$
for all $X\in\mathfrak X(M)$. If $\xi$ is closed conformal, then a straightforward computation 
shows that $\theta_{\xi}$ is a closed $1-$form.

\begin{theorem}\label{thm:deformacao de metricas Kahler por campo conforme fechado}
 Let $(M,J,g=\langle\cdot,\cdot\rangle)$ be a k\"ahlerian manifold with Levi-Civita connection $\nabla$, and $\xi\in\mathfrak X(M)$ be a closed conformal vector field on $M$. 
 If $|\xi|^2<c$ on $M$, for some positive constant $c$, and $\mu=(c-|\xi|^2)^{-1}$,
then the covariant symmetric $2-$tensor field
\begin{equation}\label{eq:deformacao de metricas Kahler por campo conforme fechado}
\tilde g=\mu g+\mu^2(\theta_{\xi}^2+\theta_{J\xi}^2)
\end{equation}
defines another k\"ahlerian metric on $(M,J)$. 
\end{theorem}

\begin{proof}
The $2-$tensor $\tilde g$ is clearly positive definite, and thus defines a riemannian metric on $M$. On the other hand, for $X,Y\in\mathfrak X(M)$, the hermitian character of $\langle\cdot,\cdot\rangle$ gives
\begin{equation}\nonumber
\begin{split}
\tilde g(JX,JY)&\,=\mu\langle JX,JY\rangle+\mu^2(\theta_{\xi}^2(X,Y)+\theta_{J\xi}^2(X,Y))\\
&\,=\mu\langle X,Y\rangle+\mu^2(\langle\xi,JX\rangle\langle\xi,JY\rangle+\langle J\xi,JX\rangle\langle J\xi,JY\rangle)\\
&\,=\mu\langle X,Y\rangle+\mu^2(\langle J\xi,X\rangle\langle J\xi,Y\rangle+\langle\xi,X\rangle\langle \xi,Y\rangle)\\
&\,=\tilde g(X,Y),
\end{split}
\end{equation}
so that $\tilde g$ is also hermitian with respect to $J$.

Next, let $\tilde\omega$ be the k\"ahlerian form of $\tilde g$. For $X,Y\in\mathfrak X(M)$, we have
\begin{equation}\nonumber
\begin{split}
\tilde\omega(X,Y)&\,=\tilde g(JX,Y)=\mu\langle JX,Y\rangle+\mu^2(\theta_{\xi}^2(JX,Y)+\theta_{J\xi}^2(JX,Y))\\
&\,=\mu\omega(X,Y)+\mu^2(\langle\xi,JX\rangle\langle\xi,Y\rangle+\langle J\xi,JX\rangle\langle J\xi,Y\rangle)\\
&\,=\mu\omega(X,Y)+\mu^2(-\theta_{J\xi}(X)\theta_{\xi}(Y)+\theta_{\xi}(X)\theta_{J\xi}(Y))\\
&\,=\mu\omega(X,Y)+\mu^2(\theta_{\xi}\wedge\theta_{J\xi})(X,Y)
\end{split}
\end{equation}
and, hence,
\begin{equation}\label{eq:auxiliar 1 para deformacao de metricas Kahler por campo conforme fechado}
\tilde\omega=\mu\omega+\mu^2\theta_{\xi}\wedge\theta_{J\xi}.
\end{equation}
Since $\omega$ e $\theta_{\xi}$ are closed, it follows that
\begin{equation}\label{eq:auxiliar 2 para deformacao de metricas Kahler por campo conforme fechado}
d\tilde\omega=d\mu\wedge\omega+2\mu d\mu\wedge\theta_{\xi}\wedge\theta_{J\xi}-\mu^2\theta_{\xi}\wedge d\theta_{J\xi}.
\end{equation}

Letting $\psi$ be the conformal factor of $\xi$ and $X\in\mathfrak X(M)$, it follows that
\begin{equation}\label{eq:auxiliar 4 para deformacao de metricas Kahler por campo conforme fechado}
\begin{split}
d\mu(X)&\,=X(\mu)=(c-|\xi|^2)^{-2}X\langle\xi,\xi\rangle\\
&\,=2\mu^2\langle\nabla_X\xi,\xi\rangle=2\mu^2\langle\psi X,\xi\rangle\\
&\,=2\psi\mu^2\theta_{\xi}(X)
\end{split}
\end{equation}
and, hence,
\begin{equation}\label{eq:auxiliar 3 para deformacao de metricas Kahler por campo conforme fechado}
d\mu=2\psi\mu^2\theta_{\xi}.
\end{equation}
Therefore, for $X,Y\in\mathfrak X(M)$, Koszul's formula and (\ref{eq:parallelism of J in Kahler manifolds}) give
\begin{equation}\nonumber
 \begin{split}
d\theta_{J\xi}(X,Y)&\,=X(\theta_{J\xi}(Y))-Y(\theta_{J\xi}(X))-\theta_{J\xi}([X,Y])\\
&\,=X\langle J\xi,Y\rangle-Y\langle J\xi,X\rangle-\langle J\xi,[X,Y]\rangle\\
&\,=\langle\nabla_XJ\xi,Y\rangle-\langle\nabla_YJ\xi,X\rangle\\
&\,=\langle J\nabla_X\xi,Y\rangle-\langle J\nabla_Y\xi,X\rangle\\
&\,=\langle J(\psi X),Y\rangle-\langle J(\psi Y),X\rangle\\
&\,=2\psi\langle JX,Y\rangle=2\psi\omega(X,Y),
 \end{split}
\end{equation}
i.e., $d\theta_{J\xi}=2\psi\omega$. Substituting the above expressions for $d\mu$ e $d\theta_{J\xi}$ in (\ref{eq:auxiliar 2 para deformacao de metricas Kahler por campo conforme fechado}) and taking into account
that $\theta_{\xi}\wedge\theta_{\xi}=0$, we finally get
$$d\tilde\omega=2\psi\mu^2\theta_{\xi}\wedge\omega+4\psi\mu^3\theta_{\xi}\wedge\theta_{\xi}\wedge\theta_{J\xi}-\mu^2\theta_{\xi}\wedge(2\psi\omega)=0.$$
\end{proof}

Our next result gives a reasonable set of conditions under which $(M,\tilde g)$ is a complete riemannian manifold. 

\begin{proposition}\label{prop:completeness of the deformed metric}
Under the hypotheses of Theorem \ref{thm:deformacao de metricas Kahler por campo conforme fechado}, suppose that the conformal factor $\psi$ of $\xi$ is bounded and does not vanish outside a compact subset of $M$. If $|\xi|^2:M\rightarrow[0,+\infty)$ is proper and such that $\sup_M|\xi|^2=c$, then $(M,\tilde g)$ is complete.
\end{proposition}

\begin{proof}
Let $\tilde\ell(\cdot)$ denote length with respect to $\tilde g$. Standard riemannian geometry assures that 
it suffices to show that if a smooth curve $\gamma:[0,+\infty)\rightarrow M$ escapes from all compact subsets of $M$, then $\tilde\ell(\gamma)=+\infty$. To this end, we start by observing that
$$\tilde g(v,v)=\mu g(v,v)+\mu^2(\langle\xi,v\rangle^2+\langle J\xi,v\rangle^2)\geq(\mu\langle\xi,v\rangle)^2.$$

Let $K\subset M$ be a compact set such that $\psi\neq 0$ on $K^c$, and $t_0>0$ such that $\gamma(t)\notin K$ for $t>t_0$. Letting $\sup_M|\psi|=\alpha<+\infty$, we estimate
\begin{equation}\nonumber
 \begin{split}
\tilde\ell(\gamma_{|[0,t]})&\,\geq\int_{t_0}^t\tilde g(\gamma'(s),\gamma'(s))^{1/2}ds\geq\int_{t_0}^t|\mu(\gamma(s))\langle\xi(\gamma(s)),\gamma'(s)\rangle| ds\\
&\,=\int_{t_0}^t\frac{1}{|\psi(\gamma(s))|}\cdot\frac{1}{c-|\xi(\gamma(s))|^2}\left|\langle\xi(\gamma(s)),\nabla_{\gamma'(s)}\xi\rangle\right| ds\\  
&\,\geq\frac{1}{2\alpha}\left|\int_{t_0}^t\frac{1}{c-|\xi(\gamma(s))|^2}\frac{d}{ds}|\xi(\gamma(s))|^2ds\right|\\  
&\,=\frac{1}{2\alpha}\left|\log(c-|\xi(\gamma(t_0))|^2)-\log(c-|\xi(\gamma(t))|^2)\right|.
 \end{split}
\end{equation}

Let $\epsilon>0$ be given. Since $|\xi|^2$ is proper, $|\xi|^2<c$ and $\sup_M|\xi|^2=c$, there exists a compact subset $L_{\epsilon}$ of $M$ such that $|\xi|^2>c-\epsilon$ in $L_{\epsilon}^c$.
Since $\gamma$ is divergent, there exists $t_{\epsilon}>t_0$ such that $\gamma(t)\in L_{\epsilon}^c$ for $t>t_{\epsilon}$. Hence, for $t>t_0,t_{\epsilon}$, the above computations give
$$\tilde\ell(\gamma_{|[0,t]})\geq\frac{1}{2\alpha}\left(\log(c-|\xi(\gamma(t_0))|^2)-\log\epsilon\right),$$
so that $\tilde\ell(\gamma)=\lim_{t\rightarrow+\infty}\tilde\ell(\gamma_{|[0,t]})=+\infty$.
\end{proof}

\begin{remark}
The above result continues to hold if we assume that (in the notations of the proof) $\psi^{-1}(0)\cap K^c$ is a set of isolated points. It suffices t
o split the trace of $\gamma$ into pieces along which $\psi\neq 0$, estimate the length of each such piece as we did above and add the results.  
\end{remark}

\begin{example}\label{exam:espaco hiperbolico complexo}
In the complex euclidean $n-$space $\mathbb C^n$, let $J$ be the standard quasi-complex structure, $g=\langle\cdot,\cdot\rangle$ be the standard metric and
$\mathbb B^n=\{z\in\mathbb C^n;\,|z|<1\}$. Since the vector field $\xi(p)=p$ is closed and conformal, we can furnish $\mathbb B^n$ with a k\"ahlerian metric $\tilde g$, such that
$$\tilde g=\frac{1}{1-|\xi|^2}\langle\cdot,\cdot\rangle+\frac{1}{(1-|\xi|^2)^2}(\theta_{\xi}^2+\theta_{J\xi}^2).$$

An immediate application of the previous proposition readily gives the completeness of $(\mathbb B^n,\tilde g)$. Therefore, the formula of Proposition 
\ref{prop:holomorphic sectional curvature of g tilde} for the holomorphic sectional curvature of $(\mathbb B^n,J,\tilde g)$, 
together with Hawley-Igusa theorem on the classification of complete, simply connected, k\"ahlerian manifolds of constant holomorphic sectional curvature
shows that $(\mathbb B^n,J,\tilde g)$ is nothing but the complex hyperbolic space $\mathbb C\mathbb H^n$. 

\end{example}

\begin{example}\label{exam:related to the main theorems}
Let $(N^{n-1},J_N,g_N)$ be a k\"ahlerian manifold (of complex dimension $n-1$) and $\mathbb T$ be a flat torus with its standard quasi-complex structure. If $M^n=N\times\mathbb T$ is furnished with the product quasi-complex structure and the product metric, then $M^n$ is a k\"ahlerian manifold. Moreover, if $\mathbb T$ is the quotient of the lattice $\mathcal L$ in $\mathbb R^2$ and $Z$ stands for the canonical vector field along one of the directions of the lattice, then $Z$ induces a parallel (hence, closed conformal) vector field on $\mathbb T$, which can be lifted to a parallel vector field on $M$. 
\end{example}

We now relate the holomorphic sectional curvatures of $(M,J,g)$ and $(M,J,\tilde g)$. To this end, we need first to relate the corresponding Levi-Civita connections, and we do so in the coming result. Along the rest of this section, we write $|\xi|^2$ for $g(\xi,\xi)$.

\begin{proposition}\label{prob:Levi-Civita connection with respect to g tilde}
Let $(M^n,J,g)$ be a k\"ahlerian manifold, $\xi\in\mathfrak X(M)$ be a closed conformal vector field such that $|\xi|^2<c$, and $\tilde g$ be the k\"ahlerian metric on $(M,J)$ given as in
Theorem \ref{thm:deformacao de metricas Kahler por campo conforme fechado}. For $X\in\mathfrak X(M)$, let $\nabla$ and $\tilde\nabla$ respectively denote the Levi-Civita connections of
$g$ and $\tilde g$. Then,
\begin{equation}\label{eq:Levi-Civita connection with respect to g tilde}
 \tilde\nabla_XY=\nabla_XY+\psi\mu\{\langle\xi,X\rangle Y+\langle\xi,Y\rangle X+\langle J\xi,X\rangle JY+\langle J\xi,Y\rangle JX\},
\end{equation}
where $\psi$ is the conformal factor of $\xi$.
\end{proposition}

\begin{proof}
As in the proof of Theorem \ref{thm:deformacao de metricas Kahler por campo conforme fechado}, let $g=\langle\cdot,\cdot\rangle$ and $|\cdot|$ be the corresponding norm. On the one hand, for $X,Y,Z\in\mathfrak X(M)$ we have
\begin{equation}\label{eq:auxiliar 1 para relacao entre conexoes}
\begin{split}
2\tilde g(\tilde\nabla_XY,Z)
&\,=2\mu\langle\tilde\nabla_XY+\mu\langle\tilde\nabla_XY,\xi\rangle\xi+\mu\langle\tilde\nabla_XY,J\xi\rangle J\xi,Z\rangle.
\end{split}
\end{equation}
On the other, it follows from Koszul's formula that
\begin{equation}\label{eq:auxiliar 2 para relacao entre conexoes}
\begin{split}
2\tilde g(\tilde\nabla_XY,Z)&\,=X(\tilde g(Y,Z))+Y(\tilde g(Z,X))-Z(\tilde g(X,Y))\\
&\,\,\,\,\,\,\,-\tilde g(X,[Y,Z])+\tilde g(Y,[Z,X])+\tilde g(Z,[X,Y]).
\end{split}
\end{equation}

We compute each summand at the right hand side of the last expression above. Taking into account (\ref{eq:auxiliar 3 para deformacao de metricas Kahler por campo conforme fechado}), the closed conformal character of $\xi$ and (\ref{eq:parallelism of J in Kahler manifolds}), we get
\begin{equation}\label{eq:auxiliar 3 para relacao entre conexoes}
\begin{split}
X(\tilde g(Y,Z))&\,=2\psi\mu^2\langle\xi,X\rangle\langle Y,Z\rangle+\mu X\langle Y,Z\rangle\\
&\,\,\,\,\,\,\,+4\psi\mu^3\langle\xi,X\rangle(\langle\xi,Y\rangle\langle\xi,Z\rangle+\langle J\xi,Y\rangle\langle J\xi,Z\rangle)\\
&\,\,\,\,\,\,\,+\mu^2(\psi\langle X,Y\rangle\langle\xi,Z\rangle+\langle\xi,\nabla_XY\rangle\langle\xi,Z\rangle)\\
&\,\,\,\,\,\,\,+\mu^2(\psi\langle\xi,Y\rangle\langle X,Z\rangle+\langle\xi,Y\rangle\langle\xi,\nabla_XZ\rangle)\\
&\,\,\,\,\,\,\,+\mu^2(\psi\langle JX,Y\rangle\langle J\xi,Z\rangle+\langle J\xi,\nabla_XY\rangle\langle J\xi,Z\rangle)\\
&\,\,\,\,\,\,\,+\mu^2(\psi\langle J\xi,Y\rangle\langle JX,Z\rangle+\langle J\xi,Y\rangle\langle J\xi,\nabla_XZ\rangle).
\end{split}
\end{equation}
Also,
\begin{equation}\label{eq:auxiliar 4 para relacao entre conexoes}
\begin{split}
\tilde g([X,Y],Z)
&\,=\mu\langle[X,Y],Z\rangle+\mu^2(\langle\xi,[X,Y]\rangle\langle\xi,Z\rangle+\langle J\xi,[X,Y]\rangle\langle J\xi,Z\rangle).
\end{split}
\end{equation}

Similar computations to those above yield corresponding formulae to the remaining summands at the right hand of (\ref{eq:auxiliar 2 para relacao entre conexoes}). Substituting all of these in (\ref{eq:auxiliar 2 para relacao entre conexoes}) and performing some cancellations, we obtain 
\begin{equation}\label{eq:auxiliar 6 para relacao entre conexoes}
\begin{split}
2\tilde g(\tilde\nabla_XY,Z)&\,=2\mu\langle\nabla_XY+\psi\mu(\langle\xi,X\rangle Y+\langle\xi,Y\rangle X)\\
&\,\,\,\,\,\,\,+2\psi\mu^2(\langle\xi,X\rangle\langle\xi,Y\rangle\xi+\langle\xi,X\rangle\langle J\xi,Y\rangle J\xi)\\
&\,\,\,\,\,\,\,+2\psi\mu^2(\langle J\xi,X\rangle\langle\xi,Y\rangle J\xi-\langle J\xi,X\rangle\langle J\xi,Y\rangle\xi)\\
&\,\,\,\,\,\,\,+\mu(\langle\nabla_XY,\xi\rangle\xi+\langle\nabla_XY,J\xi\rangle J\xi)\\
&\,\,\,\,\,\,\,+\mu(\psi\langle J\xi,Y\rangle JX+\psi\langle J\xi,X\rangle JY),Z\rangle.
\end{split}
\end{equation}
Comparing (\ref{eq:auxiliar 1 para relacao entre conexoes}) with (\ref{eq:auxiliar 6 para relacao entre conexoes}), we get
\begin{equation}\label{eq:auxiliar 7 para relacao entre conexoes}
\begin{split}
&\tilde{\nabla}_XY-\nabla_XY+\mu\langle\tilde{\nabla}_XY-\nabla_XY,\xi\rangle\xi+\mu\langle\tilde{\nabla}_XY-\nabla_XY,J\xi\rangle J\xi=\\
&\,\,\,\,=\psi\mu(\langle\xi,X\rangle Y+\langle\xi,Y\rangle X+\langle J\xi,Y\rangle JX+\langle J\xi,X\rangle JY)\\
&\,\,\,\,\,\,\ \ +2\psi\mu^2(\langle\xi,X\rangle\langle\xi,Y\rangle-\langle J\xi,X\rangle\langle J\xi,Y\rangle)\xi\\
&\,\,\,\,\,\,\ \ +2\psi\mu^2(\langle\xi,X\rangle\langle J\xi,Y\rangle+\langle J\xi,X\rangle\langle\xi,Y\rangle)J\xi.
\end{split}
\end{equation}

If we let $W=\tilde{\nabla}_XY-\nabla_XY$ and $\mu F(X,Y)$ be the summand on the right side of the equality in (\ref{eq:auxiliar 7 para relacao entre conexoes}), we obtain
\begin{equation}\label{eq:expressao para W}
W+\mu\langle W,\xi\rangle\xi+\mu\langle W,J\xi\rangle J\xi=\mu F(X,Y).
\end{equation}
Taking the inner product  of (\ref{eq:expressao para W}) with $\xi$ and $J\xi$, respectively, and recalling that $1+\mu\langle\xi,\xi\rangle=c\mu$, we find
$$\langle W,\xi\rangle=c^{-1}\langle F(X,Y),\xi\rangle\ {\rm and}\ \langle W,J\xi\rangle=c^{-1}\langle F(X,Y),J\xi\rangle.$$
Hence, (\ref{eq:expressao para W}) provides
\begin{equation}\label{eq:expressão 2 para W}
W=\mu(F(X,Y)-c^{-1}\langle F(X,Y),\xi\rangle\xi-c^{-1}\langle F(X,Y),J\xi\rangle J\xi).
\end{equation}
Finally, since $\langle J\xi,\xi\rangle=0$, we get
\begin{equation}\nonumber
\begin{split}
\langle F(X,Y),\xi\rangle&\,=2\psi(\langle\xi,X\rangle\langle\xi,Y\rangle(1+\mu\langle\xi,\xi\rangle)-\langle J\xi,X\rangle\langle J\xi,Y\rangle(1+\mu\langle\xi,\xi\rangle))\\
&\,\,\,\,=2\psi\mu c(\langle\xi,X\rangle\langle\xi,Y\rangle-\langle J\xi,X\rangle\langle J\xi,Y\rangle)
\end{split}
\end{equation}
and, similarly, 
$$\langle F(X,Y), J\xi\rangle J=2\psi\mu c(\langle\xi,X\rangle\langle J\xi,Y\rangle+\langle\xi,Y\rangle\langle J\xi,X\rangle).$$
Therefore, (\ref{eq:auxiliar 7 para relacao entre conexoes}) and (\ref{eq:expressão 2 para W}) yield
$$\tilde{\nabla}_XY-\nabla_XY=W=\psi\mu\{\langle\xi,X\rangle Y+\langle\xi,Y\rangle X+\langle J\xi,Y\rangle JX+\langle J\xi,X\rangle JY\}.$$
\end{proof}

Before we can proceed to relate the holomorphic sectional curvatures of $g=\langle\cdot,\cdot\rangle$ and $\tilde g$, we need a few more preliminaries. 

Firstly, the closed conformal condition on $\xi\in\mathfrak X(M^n)$ (recall that $n$ stands for the {\em complex} dimension of $M$) readily gives
$$\psi=\frac{1}{2n}\text{div}\,\xi,$$
and Lemma $1$ of \cite{Urbano:98} shows that
$$|\xi|^2\nabla(\text{div}\,\xi)=-2n\text{Ric}(\xi)\xi,$$
where $\nabla(\text{div}\,\xi)$ stands for the gradient of the divergence of $\xi$ with respect to $g$ and $\text{Ric}(\xi)$ for the normalized Ricci curvature of $(M,g)$ in the direction of $\xi$. In particular, at each point where $\xi\neq 0$, we get
\begin{equation}\label{eq:gradiente de psi onde xi nao se anula}
\nabla\psi=-\text{Ric}(\hat\xi)\xi,
\end{equation}
where $\hat\xi=\frac{\xi}{|\xi|}$. 

Secondly, it follows from \eqref{eq:auxiliar 4 para deformacao de metricas Kahler por campo conforme fechado} that
\begin{equation}\label{eq:derivadas de mu}
X(\mu)=2\psi\mu^2\langle X,\xi\rangle\,\,\text{and}\,\,JX(\mu)=-2\psi\mu^2\langle X,J\xi\rangle.
\end{equation}
Therefore, at each point where $\xi$ does not vanish, \eqref{eq:gradiente de psi onde xi nao se anula} furnishes
\begin{equation}\label{eq:derivada 1 de psi}
X(\psi)=\langle X,\nabla\psi\rangle=-\text{Ric}(\hat\xi)\langle X,\xi\rangle
\end{equation}
and, analogously,
\begin{equation}\label{eq:derivada 2 de psi} 
JX(\psi)=-\text{Ric}(\hat\xi)\langle JX,\xi\rangle=\text{Ric}(\hat\xi)\langle X,J\xi\rangle.
\end{equation}

Thirdly, if $\langle X,X\rangle=1$, then, since $1+\mu\langle\xi,\xi\rangle=c\mu$, we get
\begin{equation}\nonumber
\begin{split}
\tilde{g}(X,\xi)=\mu\langle X,\xi\rangle+\mu^2(\langle X,\xi\rangle\langle\xi,\xi\rangle+\langle X,J\xi\rangle\langle\xi,J\xi\rangle)=c\mu^2\langle X,\xi\rangle,
\end{split}
\end{equation}
$$\tilde{g}(X,J\xi)=-\tilde{g}(JX,\xi)=-c\mu^2\langle JX,\xi\rangle=c\mu^2\langle X,J\xi\rangle$$
and
\begin{equation}\nonumber
\begin{split}
\tilde{g}(X,X)
&\,=\mu+\mu^2(\langle X,\xi\rangle^2+\langle X,J\xi\rangle^2).
\end{split}
\end{equation}

We shall also need to use {\em hermitian geodesic frames}, in the sense of the following lemma.

\begin{lemma}\label{lemma:referencial geodesico em variedade Kahler}
If $(M^n,J,g)$ is a k\"ahlerian manifold and $p\in M$, then there exist a neighborhood $U\subset M$ of $p$ and a hermitian frame field in $U$ which is geodesic at $p$.
\end{lemma}

\begin{proof}
Take a normal ball $U\subset M$ centered at $p$ and a hermitian basis $(e_1,J_pe_1,\ldots,$ $e_n,J_pe_n)$ for $T_pM$. By parallel translating such vectors along the 
geodesic rays on $U$ departing from $p$, we get 
an orthonormal frame field $(e_1,e_1',\ldots,e_n,e_n')$ in $U$, which is geodesic at $p$. We assert that $e_k'=Je_k$ in $U$, for $1\leq k\leq n$. In fact, given $q\in U$, take the radial geodesic $\gamma:[0,1]\rightarrow U$, such that $\gamma(0)=p$ and $\gamma(1)=q$;
it follows from (\ref{eq:parallelism of J in Kahler manifolds}) that
$$\frac{D}{dt}Je_k=\nabla_{\gamma'}Je_k=J\nabla_{\gamma'}e_k=0,$$
so that $Je_k$ is parallel along $\gamma$. However, since $e_k'(p)=J_pe_k=(Je_k)(p)$, uniqueness of parallel transport gives $e_k'=Je_k$ along $\gamma$, so that $e_k'=Je_k$ at $q$. 
\end{proof}

We are finally in position to state and prove the following

\begin{proposition}\label{prop:holomorphic sectional curvature of g tilde}
Let $(M^n,J,g)$ be a k\"ahlerian manifold, $\xi\in\mathfrak X(M)$ be a closed conformal vector field with isolated zeros and such that $|\xi|^2<c$ on $M$, $\psi$ be the conformal factor of $\xi$ and $\tilde g$ be the k\"ahlerian metric on $(M,J)$ given as in Theorem \ref{thm:deformacao de metricas Kahler por campo conforme fechado}. For $X\in T_pM$ unitary with respect to $g$, let $K(X)$ and $\tilde K(X)$ denote the holomorphic sectional curvatures of $(M,J,g)$ and $(M,J,\tilde g)$ with respect to $X$, respectively. Then, at each point where $\xi$ does not vanish, we have
\begin{equation}\nonumber
\begin{split}
\tilde K(X)&\,=\frac{1}{\tilde{g}(X,X)^2}\{\mu K(X)+\mu^2{\rm Ric}(\hat\xi)(\langle X,\xi\rangle^2+\langle X,J\xi\rangle^2)\}\\
&\,\,\,\,\,\,\,+\frac{1}{\tilde{g}(X,X)}\cdot 2\mu{\rm Ric}(\hat\xi)(\langle X,\xi\rangle^2+\langle X,J\xi\rangle^2)-4\psi^2,
\end{split}
\end{equation}
where ${\rm Ric}(\hat\xi)$ stands for the normalized Ricci curvature of $g$ in the direction of $\hat\xi$
\end{proposition}

\begin{proof}
Extend $X$ to a smooth vector field around $p$. 
If $\tilde{R}$ denotes the curvature tensor of $(M,\tilde{g})$, the holomorphic sectional curvature of $(M,\tilde{g})$ with respect to $X$ is given by
\begin{equation}\nonumber
\begin{split}
\tilde K(X)&\,=\frac{\tilde{g}(\tilde{R}(X,JX)JX,X)}{\tilde{g}(X,X)\tilde{g}(JX,JX)-\tilde{g}(X,JX)^2}\\
&\,=\frac{1}{\tilde{g}(X,X)^2}\tilde{g}(\tilde{\nabla}_X\tilde{\nabla}_{JX}JX-\tilde{\nabla}_{JX}\tilde{\nabla}_X JX-\tilde{\nabla}_{[X,JX]}JX,X)\\
&\,=\frac{1}{\tilde{g}(X,X)^2}\{X(\tilde{g}(\tilde{\nabla}_{JX}JX,X))-\tilde{g}(\tilde{\nabla}_{JX}JX,\tilde{\nabla}_{X}X)-JX(\tilde{g}(\tilde{\nabla}_{X}JX,X))\\
&\hspace{2cm}+\tilde{g}(\tilde{\nabla}_{X}JX,\tilde{\nabla}_{JX}X)-\tilde{g}(\tilde{\nabla}_{[X,JX]}JX,X)\}.
\end{split}
\end{equation}

In order to simplify the notation, frow now on we let
\begin{equation}\label{eq:alpha and beta}
\alpha=2\psi\mu\langle X,\xi\rangle\ \ \text{and}\ \ \beta=2\psi\mu\langle X,J\xi\rangle.
\end{equation}
It easily follows from (\ref{eq:Levi-Civita connection with respect to g tilde}) that
\begin{equation}\label{eq:auxiliar para mudanca da curvatura seccional holomorfa}
\begin{split}
\tilde{\nabla}_XX&\,=\nabla_XX+2\psi\mu(\langle X,\xi\rangle X+\langle X,J\xi\rangle JX)\\
&\,=\nabla_X X+\alpha X+\beta JX,\\
\tilde{\nabla}_XJX&\,=J\tilde{\nabla}_XX=\nabla_XJX+\alpha JX-\beta X,\\
\tilde{\nabla}_{JX}X&\,=\nabla_{JX}X+2\psi\mu(-\langle X,J\xi\rangle X+\langle X,\xi\rangle JX)\\
&\,=\nabla_{JX}X-\beta X+\alpha JX,\\
\tilde{\nabla}_{JX}JX&\,=J\tilde{\nabla}_{JX}X=\nabla_{JX}JX-\alpha X-\beta JX.
\end{split}
\end{equation}

For the subsequent computations we observe that, since the value of $\tilde K(X)$ at $p$ depends only on the value 
of $X$ at $p$, we can assume $(\nabla_vX)(p)=0$ and $(\nabla_vJX)(p)=0$ for all $v\in T_pM$ (just apply Lemma 
\ref{lemma:referencial geodesico em variedade Kahler} to obtain a hermitian geodesic frame $(e_1,Je_1,\ldots,e_n,Je_n)$ around $p$, such that $e_1(p)=X_p$). Hence, $[X,JX]=0$ at $p$, so that the last term in the expression for $\tilde{K}(X)$ vanishes at $p$.

With the aid of the formulas in \eqref{eq:auxiliar para mudanca da curvatura seccional holomorfa}, we then get at the point $p$
\begin{equation}\label{eq:curvatura seccional holomorfa}
\begin{split}
\tilde{K}(X)&\,=\frac{1}{\tilde{g}(X,X)^2}\{X(\tilde{g}(\nabla_{JX}JX-\alpha X-\beta JX,X))\\
&\hspace{2cm}-\tilde{g}(\nabla_{JX}JX-\alpha X-\beta JX,\nabla_{X}X+\alpha X+\beta JX)\\
&\hspace{2cm}-JX(\tilde{g}(\nabla_{X}JX-\beta X+\alpha JX,X))\\
&\hspace{2cm}+\tilde{g}(\nabla_{X}JX-\beta X+\alpha JX,\nabla_{JX}X-\beta X+\alpha JX\}\\
&\,=\frac{1}{\tilde{g}(X,X)^2}\{\underbrace{X(\tilde{g}(\nabla_{JX}JX,X))-JX(\tilde{g}(\nabla_{X}JX,X))}_I\}\\
&\,\,\,\,\,\,\,+\frac{1}{\tilde{g}(X,X)}\{\underbrace{-X(\alpha)+2(\alpha^2+\beta^2)+JX(\beta)}_{II}\}\\
&\,\,\,\,\,\,\,+\frac{1}{\tilde{g}(X,X)^2}\{\underbrace{-\alpha X(\tilde{g}(X,X))+\beta JX(\tilde{g}(X,X))}_{III}\}.
\end{split}
\end{equation}

Let us separately compute each of $I$, $II$ and $III$:
\begin{equation}\nonumber
\begin{split}
I&\,=X(\mu)\langle\nabla_{JX}JX,X\rangle+\mu X\langle\nabla_{JX}JX,X\rangle\\
&\,\,\,\,\,\,+X(\mu^2)(\langle\nabla_{JX}JX,\xi\rangle\langle X,\xi\rangle+\langle\nabla_{JX}JX,J\xi\rangle\langle X,J\xi\rangle)\\
&\,\,\,\,\,\,+\mu^2X(\langle\nabla_{JX}JX,\xi\rangle\langle X,\xi\rangle+\langle\nabla_{JX}JX,J\xi\rangle\langle X,J\xi\rangle)\\
&\,\,\,\,\,\,-JX(\mu)\langle\nabla_{X}JX,X\rangle-\mu JX\langle\nabla_{X}JX,X\rangle\\
&\,\,\,\,\,\,-JX(\mu^2)(\langle\nabla_{X}JX,\xi\rangle\langle X,\xi\rangle+\langle\nabla_{X}JX,J\xi\rangle\langle X,J\xi\rangle)\\
&\,\,\,\,\,\,-\mu^2JX(\langle\nabla_{X}JX,\xi\rangle\langle X,\xi\rangle+\langle\nabla_{X}JX,J\xi\rangle\langle X,J\xi\rangle).
\end{split}
\end{equation}

Recalling that $\nabla_vX=\nabla_vJX=0$ and $[X,JX]=0$ at $p$, we get at this point
\begin{equation}\label{eq:expressao para o termo I}
\begin{split}
I&\,=\mu\langle\nabla_X\nabla_{JX}JX,X\rangle+\mu^2\langle\nabla_X\nabla_{JX}JX,\xi\rangle\langle X,\xi\rangle\\
&\,\,\,\,\,\,+\mu^2\langle\nabla_X\nabla_{JX}JX,J\xi\rangle\langle X,J\xi\rangle
-\mu\langle\nabla_{JX}\nabla_{X}JX,X\rangle\\
&\,\,\,\,\,\,-\mu^2\langle\nabla_{JX}\nabla_{X}JX,\xi\rangle\langle X,\xi\rangle-\mu^2\langle\nabla_{JX}\nabla_{X}JX,J\xi\rangle\langle X,J\xi\rangle\\
&\,=\mu\langle R(X,JX)JX,X\rangle+\mu^2\langle R(X,JX)JX,\xi\rangle\langle X,\xi\rangle\\
&\,\,\,\,\,\,+\mu^2\langle R(X,JX)JX,J\xi\rangle\langle X,J\xi\rangle\\
&\,=\mu K(X)+\mu^2\langle R(X,JX)JX,\xi\rangle\langle X,\xi\rangle+\mu^2\langle R(X,JX)JX,J\xi\rangle\langle X,J\xi\rangle.
\end{split}
\end{equation}

We now turn our attention to the computation of $II$. By \eqref{eq:alpha and beta}, \eqref{eq:derivadas de mu} and the closed conformal character of $\xi$, we obtain
\begin{equation}\nonumber
\begin{split}
II&\,=-2X(\psi\mu\langle X,\xi\rangle)+2[(2\psi\mu\langle X,\xi\rangle)^2+(2\psi\mu\langle X,J\xi\rangle)^2]+2JX(\psi\mu\langle X,J\xi\rangle)\\
&\,=-2X(\psi)\mu\langle X,\xi\rangle+2JX(\psi)\mu\langle X,J\xi\rangle+4\psi^2\mu^2(\langle X,\xi\rangle^2+\langle X,J\xi\rangle^2)-4\psi^2\mu.
\end{split}
\end{equation}

If $p\in M$ is such that $\xi(p)\neq 0$, it follows from (\ref{eq:derivada 1 de psi}) and (\ref{eq:derivada 2 de psi}) that
\begin{equation}\label{eq:expressao para o termo II}
\begin{split}
II&\,=2\mu{\rm Ric}(\hat\xi)\langle X,\xi\rangle^2+2\mu{\rm Ric}(\hat\xi)\langle X,J\xi\rangle^2\\
&\,\,\,\,\,\,\,\,+4\psi^2\left(\mu+\mu^2(\langle X,\xi\rangle^2+\langle X,J\xi\rangle^2)\right)-8\psi^2\mu\\
&\,=2\mu{\rm Ric}(\hat\xi)(\langle X,\xi\rangle^2+\langle X,J\xi\rangle^2)+4\psi^2\tilde{g}(X,X)-8\psi^2\mu.
\end{split}
\end{equation}

In order to deal with $III$, we recall once more that $(\nabla_vX)(p)=0$ for all $v\in T_pM$. 
Then, at the point $p$ and for $Y\in\mathfrak X(M)$, we get from \eqref{eq:derivadas de mu} and the closed conformal character of $\xi$ that
\begin{equation}\nonumber
\begin{split}
Y(\tilde{g}(X,X))&\,=Y(\mu)+Y(\mu^2)(\langle X,\xi\rangle^2+\langle X,J\xi\rangle^2)+\mu^2(Y(\langle X,\xi\rangle^2)+Y(\langle X,J\xi\rangle^2))\\
&\,=2\psi\mu^2\langle Y,\xi\rangle+4\psi\mu^3\langle Y,\xi\rangle(\langle X,\xi\rangle^2+\langle X,J\xi\rangle^2)\\
&\,\,\,\,\,\,\,+2\mu^2(\langle X,\xi\rangle\langle X,\psi Y\rangle+\langle X,J\xi\rangle\langle X,J(\psi Y)\rangle)\\
&\,=4\psi\mu\langle Y,\xi\rangle\tilde{g}(X,X)+2\psi\mu^2(\langle X,\xi\rangle\langle X,Y\rangle+\langle X,J\xi\rangle\langle X,JY\rangle)\\
&\,\,\,\,\,\,\,-2\psi\mu^2\langle Y,\xi\rangle.
\end{split}
\end{equation}
In particular,
\begin{equation}\label{eq:expressao para o termo III}
\begin{split}
III&\,=-2\psi\mu\langle X,\xi\rangle\{4\psi\mu\langle X,\xi\rangle\tilde g(X,X)+2\psi\mu^2\langle X,\xi\rangle-2\psi\mu^2\langle X,\xi\rangle\}\\
&\,\,\,\,\,\,\,+2\psi\mu\langle X,J\xi\rangle\{4\psi\mu\langle JX,\xi\rangle\tilde{g}(X,X)-2\psi\mu^2\langle X,J\xi\rangle-2\psi\mu^2\langle JX,\xi\rangle\}\\
&\,=-8\psi^2\mu^2(\langle X,\xi\rangle^2+\langle X,J\xi\rangle^2)\tilde{g}(X,X).
\end{split}
\end{equation}

Substituting (\ref{eq:expressao para o termo I}), (\ref{eq:expressao para o termo II}) and (\ref{eq:expressao para o termo III}) in (\ref{eq:curvatura seccional holomorfa}), we obtain
\begin{equation}\nonumber
\begin{split}
\tilde{K}(X)&\,=\frac{1}{\tilde{g}(X,X)^2}\{\mu K(X)+\mu^2\langle R(X,JX)JX,\xi\rangle\langle X,\xi\rangle\\
&\hspace{2cm}+\mu^2\langle R(X,JX)JX,J\xi\rangle\langle X,J\xi\rangle\}\\
&\,\,\,\,\,\,\,+\frac{1}{\tilde g(X,X)}\cdot 2\mu{\rm Ric}(\hat\xi)(\langle X,\xi\rangle^2+\langle X,J\xi\rangle^2)-4\psi^2.
\end{split}
\end{equation}

Now, by invoking Lemma $1$ of \cite{Urbano:98} once more, we have
\begin{equation}\nonumber
\begin{split}
\langle R(X,JX)JX,\xi\rangle&\,=-{\rm Ric}(\hat\xi)(\langle JX,\xi\rangle\langle X,JX\rangle-\langle X,\xi\rangle\langle JX,JX\rangle)\\
&\,={\rm Ric}(\hat\xi)\langle X,\xi\rangle,
\end{split}
\end{equation}
and similarly $\langle R(X,JX)JX,J\xi\rangle={\rm Ric}(\hat\xi)\langle X,J\xi\rangle$. Substituting these formulas in the last expression above for $\tilde K(X)$, we finally arrive at the formula displayed in the statement
of the theorem.
\end{proof}

\begin{remark}\label{rem:zeros isolados}
If a k\"ahlerian manifold $M^n$ has complex dimension $n>1$, and possesses a closed conformal vector field $\xi\in\mathfrak X(M)\setminus\{0\}$, then Lemma $1$ of \cite{Urbano:98} assures that the zeros of $\xi$ are all isolated. Thus, the formula for $\tilde K(X)$ holds on all of $M$, provided we substitute $\text{\rm Ric}(\hat\xi)$ by at the singular points of $\xi$. Indeed, in the notations of \eqref{eq:expressao para o termo II}, if $\xi(p)=0$, take a sequence $(p_j)_{j\geq 1}$ in $M$, converging to $p$ and such that $\xi(p_j)\neq 0$. Since $(\nabla\psi)(p_j)=-{\rm Ric}(\hat\xi(p_j))\xi(p_j)$, we get
$$\nabla\psi(p)=\lim_{j}\nabla\psi(p_j)=-\lim_j\underbrace{{\rm Ric}(\hat\xi(p_j))\frac{\xi(p_j)}{|\xi(p_j)|}}_{\text{bounded}}.|\xi(p_j)|\stackrel{j}{\longrightarrow}0.$$
Therefore, if we take ${\rm Ric}(\hat\xi)=0$ at $p$, then (\ref{eq:expressao para o termo II}) will be valid in all cases.
\end{remark}

The corollary below extends, to a general deformation $g\mapsto\tilde g$ as above, the phenomenon of holomorphic sectional curvature decay that takes place when we pass from $\mathbb C^n$ to $\mathbb C\mathbb H^n$.

\begin{corollary}\label{coro:sectional holomorphic curvature decay}
Let $(M,J,g)$ be a k\"ahlerian manifold, $\xi\in\mathfrak X(M)$ be a closed conformal vector field with isolated zeros and such that $|\xi|^2<c$ on $M$, and $\tilde g$ be the k\"ahlerian metric on $(M,J)$ given as in Theorem \ref{thm:deformacao de metricas Kahler por campo conforme fechado}. 
For $X\in T_pM$ unitary with respect to $g$ and such that $K(X)\geq 0$, we have:
\begin{enumerate}
 \item[(a)] If $X\bot\xi,J\xi$, then $\tilde K(X)\leq cK(X)-4\psi^2$.
 \item[(b)] For a general $X$, and if $M$ has nonnegative Ricci curvature, then 
 $$\tilde K(X)\leq cK(X)+2c\,\text{\rm Ric}(\hat\xi)-4\psi^2,$$
 with $\text{\rm Ric}(\hat\xi)$ being taken as $0$ if $\xi(p)=0$.
\end{enumerate}
\end{corollary}

\begin{proof}
If $X\bot\xi,J\xi$, then $\tilde g(X,X)=\mu$. Therefore, our previous result gives
$$\tilde K(X)=(c-|\xi|^2)K(X)-4\psi^2\leq cK(X)-4\psi^2.$$

For a general $X\in T_pM$ unitary, let $A=\langle X,\xi\rangle^2+\langle X,J\xi\rangle^2$ and write
\begin{equation}\nonumber
\begin{split}
\tilde{K}(X)&\,=\frac{\mu}{\tilde{g}(X,X)^2}\cdot K(X)+\frac{(\mu^2+2\mu\tilde g(X,X))A}{\tilde g(X,X)^2}\cdot{\rm Ric}(\hat\xi)-4\psi^2,
\end{split}
\end{equation}
For the first summand, note that
$$1+\mu A=\frac{c-|\xi|^2+\langle X,\xi\rangle^2+\langle JX,\xi\rangle^2}{c-|\xi|^2}\geq\frac{c-|\xi|^2}{c-|\xi|^2}=1;$$
hence, 
$$\frac{\mu}{\tilde g(X,X)^2}=\frac{1}{\mu(1+\mu A)^2}\leq\frac{1}{\mu}\leq c.$$
For the second summand, substituting $\tilde g(X,X)=\mu+\mu^2A$ we get
$$\frac{(\mu^2+2\mu\tilde g(X,X))A}{\tilde g(X,X)^2}=\frac{(3+2\mu A)A}{(1+\mu A)^2}=\frac{1}{\mu}\cdot\frac{3y+2y^2}{1+2y+y^2},$$
where $y=\mu A$. It now suffices to observe that $\frac{1}{\mu}\leq c$ and (since $y\geq 0$)
$$\frac{3y+2y^2}{1+2y+y^2}=2-\frac{1}{y+1}-\frac{1}{(y+1)^2}<2.$$
\end{proof}

\section{Two rigidity results for k\"ahlerian manifolds}

We now use the metric deformation discussed in the previous section to study the structure of connected complete k\"ahlerian manifolds furnished with a nontrivial closed conformal vector field. We start by looking at the compact case, for which we shall need the following preliminary result.

\begin{lemma}\label{lemma:volume element}
Let $(M^1,g,J)$ be a compact k\"ahlerian surface furnished with a closed conformal vector field $\xi\in\mathfrak X(M)$, with isolated set of zeros and conformal factor $\psi$. Let $c>0$ be such that $c>\max_M|\xi|^2$ and $\tilde g$ be the metric on $M$ defined as in \eqref{eq:deformacao de metricas Kahler por campo conforme fechado}. If $d_gM$ and $d_{\tilde g}M$ stand for the volume elements of $M$ with respect to $g$ and $\tilde g$, respectively, then
\begin{equation}\label{eq:volume element}
d_{\tilde g}M=c\mu^2d_gM.
\end{equation}
\end{lemma}

\begin{proof}
Since the singular points of $\xi$ are isolated, it suffices to deduce \eqref{eq:volume element} at the points where $\xi\neq 0$. At such a point $p$, if $e_1=\frac{\xi}{\tilde g(\xi,\xi)^{1/2}}=\frac{\xi}{\sqrt c\mu|\xi|}$ and $e_2=\frac{J\xi}{\tilde g(\xi,\xi)^{1/2}}=\frac{J\xi}{\sqrt c\mu|\xi|}$, then $\{e_1,e_2\}$ is a positive orthonormal basis of $T_pM$ with respect to $\tilde g$. Hence, if $\tilde\theta_1$, $\tilde\theta_2$ is the metrically dual basis, then $d_{\tilde g}M=\tilde\theta_1\wedge\tilde\theta_2$ at $p$. 

For a vector $v\in T_pM$, we have
\begin{equation}\nonumber
 \begin{split}
\tilde\theta_1(v)&\,=\frac{1}{\sqrt c\mu|\xi|}\tilde g(v,\xi)=\frac{1}{\sqrt c\mu|\xi|}\big(\mu\langle v,\xi\rangle+\mu^2\theta_{\xi}^2(v,\xi)\big)\\
&\,=\frac{1}{\sqrt c\mu|\xi|}\big(\mu\theta_{\xi}(v)+\mu^2\theta_{\xi}(v)|\xi|^2\big)\\
&\,=\frac{1}{\sqrt c\mu|\xi|}\big(\mu+\mu^2|\xi|^2\big)\theta_{\xi}(v)\\
&\,=\frac{\sqrt c\mu}{|\xi|}\theta_{\xi}(v)=\sqrt c\mu\theta_{\hat\xi}(v),
 \end{split}
\end{equation}
so that $\tilde\theta_1=\sqrt c\mu\theta_{\hat\xi}$. Likewise, $\tilde\theta_2=\sqrt c\mu\theta_{J\hat\xi}$.

However, since $\{\hat\xi,J\hat\xi\}$ is a positive orthonormal basis for $T_pM$ with respect to $g$, we have 
$$d_{\tilde g}M=\tilde\theta_1\wedge\tilde\theta_2=c\mu^2\theta_{\hat\xi}\wedge\theta_{J\hat\xi}=c\mu^2d_gM.$$
\end{proof}

\begin{theorem}\label{thm:the compact case}
Let $(M^1,g,J)$ be a connected, compact k\"ahlerian surface with gaussian curvature $K\leq 0$. If $M$ possesses a closed conformal vector field $\xi\in\mathfrak X(M)$ whose zeros are all isolated, then $K\equiv 0$, $\xi$ is parallel and $M$ is isometric to a flat torus.
\end{theorem}

\begin{proof}
As before, let $\psi$ be the conformal factor of $\xi$, choose a real number $c>0$ such that $c>\max_M|\xi|^2$ and let $\tilde g$ be defined as in \eqref{eq:deformacao de metricas Kahler por campo conforme fechado}.

Since $\xi^{-1}(0)$ is a set of isolated points and $M$ is compact, we conclude that $\xi^{-1}(0)$ is finite. Therefore, if $\tilde K$ stands for the gaussian curvature of $(M,\tilde g)$, then, at every point of $M\setminus\xi^{-1}(0)$ and in the notations of Proposition \ref{prop:holomorphic sectional curvature of g tilde}, we have 
$$K=K(\hat\xi)\ \ \text{and}\ \ \tilde K=\tilde K(\hat\xi).$$
That result also furnishes
\begin{equation}\nonumber
\begin{split}
\tilde K&\,=\frac{1}{\tilde g(\hat\xi,\hat\xi)^2}\{\mu K+\mu^2K(\langle\hat\xi,\xi\rangle^2+\langle\hat\xi,J\xi\rangle^2)\}\\
&\,\,\,\,\,\,\,+\frac{1}{\tilde g(\hat\xi,\hat\xi)}\cdot 2\mu K(\langle\hat\xi,\xi\rangle^2+\langle\hat\xi,J\xi\rangle^2)-4\psi^2.
\end{split}
\end{equation}

Now, since $\langle\hat\xi,J\xi\rangle=0$ and 
\begin{equation}\nonumber 
 \begin{split}
\tilde g(\hat\xi,\hat\xi)&\,=\mu\langle\hat\xi,\hat\xi\rangle+\mu^2\big(\theta_{\xi}^2(\hat\xi,\hat\xi\big)+\theta_{J\xi}^2(\hat\xi,\hat\xi\big)\\
&\,=\mu+\mu^2\langle\hat\xi,\xi\rangle^2=\mu+\mu^2|\xi|^2=c\mu^2,
 \end{split}
\end{equation}
we obtain, after some simple algebra,
\begin{equation}\nonumber
\begin{split}
\tilde K&\,=\frac{1}{c^2\mu^4}\big(\mu+\mu^2|\xi|^2\big)K+\frac{2\mu K}{c\mu^2}|\xi|^2-4\psi^2\\
&\,=\Big(\frac{c+|\xi|^2}{c\mu}\Big)K-4\psi^2.
\end{split}
\end{equation}
By continuity, this last formula relating $K$ and $\tilde K$ holds in all of $M$, and not just in $M\setminus\xi^{-1}(0)$. 

We now apply Gauss-Bonnet Theorem twice, with the aid of \eqref{eq:volume element}:
\begin{equation}\nonumber 
 \begin{split}
2\pi\mathcal X(M)&\,=\int_M\tilde Kd_{\tilde g}M=\int_M\Big[\Big(\frac{c+|\xi|^2}{c\mu}\Big)-4\psi^2\Big]c\mu^2d_gM\\
&\,=\int_M\Big(\frac{c+|\xi|^2}{c-|\xi|^2}\Big)Kd_gM-4c\int_M\psi^2\mu^2d_gM\\
&\,\leq\int_MKd_gM=2\pi\mathcal X(M).
\end{split}
\end{equation}

Thus, the inequality above must be an equality, which implies $K\equiv 0$ and $\psi\equiv 0$ and, in turn, $\mathcal X(M)=0$. This means that $M$ is diffeomorphic to a torus and $\xi$ is parallel. Also, since $\nabla J=0$, we get that $J\xi$ is also parallel. Therefore, $|\xi|=|J\xi|$ are constant on $M$, and since $\xi$ is nontrivial, both these vector fields do not vanish on $M$. 

Since $M$ is diffeomorphic to a torus, a theorem of Cartan assures the existence of closed geodesics $\gamma_1$ and $\gamma_2$ in $M$, representing the free homotopy classes of a set of generators of $\pi_1(M)$. Letting the flow of $\xi$ act (by isometries, since $\xi$ is parallel) on $\gamma_1$ and $\gamma_2$, we can assume that both of them start and end at the same point $p$ of $M$, so that $\pi_1(M;p)$ is generated by $[\gamma_1]$ and $[\gamma_2]$. If $\gamma_1'(0)=\alpha\xi+\beta J\xi$, for some $\alpha,\beta>0$, then the parallelisms of $\xi$ and $J\xi$ give that $\gamma_1'=\alpha\xi+\beta J\xi$ along $\gamma_1$, so that $\gamma_1$ is a geodesic loop based at $p$. Accordingly, so is $\gamma_2$.

Since $K\equiv 0$ and $K$ is diffeomorphic to a torus, its universal covering, furnished with the covering metric, is $\mathbb R^2$ with its standard flat metric. Let $\pi:\mathbb R^2\rightarrow M$ be the covering map, $o\in\pi^{-1}(p)$ and $\tilde\gamma_1$ and $\tilde\gamma_2$ be the liftings of $\gamma_1$ and $\gamma_1$ passing through $p$, respectively. Since $\pi$ is a local isometry, $\tilde\gamma_1$ and $\tilde\gamma_2$ are straight lines through $o$. Letting $\tilde\gamma_1(0)=\tilde\gamma_2(0)=o$, we can assume that $d\pi_o$ maps $\tilde\gamma_j'(0)$ to $\gamma_j'(0)$, for $j=1,2$, as well as that $\angle(\tilde\gamma_1'(0),\tilde\gamma_2'(0))=\angle(\gamma_1'(0),\gamma_2'(0))$ as oriented angles. Suppose further (also without loss of generality) that $\gamma_1$ and $\gamma_2$ (and so $\tilde\gamma_1$ and $\tilde\gamma_2$) are normalized, and let $\ell(\gamma_j)=a_j$, for $j=1,2$. If $\overrightarrow u_1$ and $\overrightarrow u_2$ are the vectors in $\mathbb R^2$ that go from $o$ to $\tilde\gamma_1(a_1)$ and $\tilde\gamma_2(a_2)$, respectively, and $\mathcal L$ is the lattice in $\mathbb R^2$ generated by $\overrightarrow u_1$ and $\overrightarrow u_2$, then a standard covering argument shows that $M$ is isometric to the flat torus $\mathbb R^2/\text{Deck}(\pi)$, quotient of the fundamental domain of $\mathcal L$ by $\pi$.

\end{proof}

\begin{remark}
The nonpositivity of the holomorphic sectional curvature and the compactness of $M$ cannot be relaxed, as shown by the examples of $\mathbb S^2$ and $\mathbb R^2$, both of which possess nonparallel closed conformal vector fields. 
\end{remark}

We now turn the case of a complete k\"ahlerian manifold $M^n$ of complex dimension $n>1$, and first recall (cf. Remark \ref{rem:zeros isolados}) thet if it has a closed conformal vector field $\xi\in\mathfrak X(M)\setminus\{0\}$, then the zeros of $\xi$ are all isolated. 

\begin{lemma}\label{lemma:folhas sao totalmente geodesicas}
Let $n>1$ and $(M^n,g,J)$ be a k\"ahlerian manifold furnished with a closed conformal vector field $\xi\in\mathfrak X(M)\setminus\{0\}$. Then, the distribution $\mathcal D$ generated by $\xi$ and $J\xi$ is involutive in $M\setminus\xi^{-1}(0)$, and its leaves are totally geodesic and k\"ahlerian in the induced metric.
\end{lemma}

\begin{proof}
It follows from \eqref{eq:parallelism of J in Kahler manifolds} and the closed conformal character of $\xi$ that
$$[\xi,J\xi]=\nabla_{\xi}J\xi-\nabla_{J\xi}\xi=J\nabla_{\xi}\xi-\psi J\xi=0.$$
Moreover, if $\Sigma$ is a leaf of $\mathcal D$ and we let $\mathcal N$ denote the Nijenhuis tensor of $\Sigma$, it is immediate to see that
$\mathcal N=0$, so that $\Sigma$ is a complex curve in $(M^n,J,g)$ and the restriction of $g$ to $\Sigma$ (which we shall also call $g$) is k\"ahlerian. 

Now, let $\alpha$ be the second fundamental form of $\Sigma$ and $(\cdot)^{\bot}$ denote orthogonal projection onto $T(\Sigma,g)^{\bot}$. Then,
$$\alpha(\xi,J\xi)=(\nabla_{\xi}J\xi)^{\bot}=(J\nabla_{\xi}\xi)^{\bot}=(\psi J\xi)^{\bot}=0.$$
\end{proof}

\begin{lemma}\label{lemma:anulamento do fator conforme}
Let $n>1$ and $(M^n,g,J)$ be a connected k\"ahlerian manifold furnished with a closed conformal vector field $\xi\in\mathfrak X(M)\setminus\{0\}$. If there exists a nontrivial leaf $\Sigma$ of the distribution $\mathcal D$ generated by $\xi$ and $J\xi$ such that the conformal factor of $\xi$ vanishes on $\Sigma$, then it vanishes on $M$. In particular, $\xi$ and $J\xi$ are parallel and do not vanish at any point of $M$.
\end{lemma}

\begin{proof}
By hypothesis, we have $\xi\neq 0$ on $\Sigma$. For $p\in\Sigma$, take a neighborhood $V$ of $p$ in $M\setminus\xi^{-1}(0)$ such that there exists in $V$ a smooth vector field $X$, with $X\bot\xi,J\xi$. Letting $\psi$ stand for the conformal factor of $\xi$ and $\alpha$ be an integral curve of $X$ starting at $q\in V\cap\Sigma$, it follows from \eqref{eq:gradiente de psi onde xi nao se anula} that
$$\frac{d}{dt}\psi(\alpha(t))=\langle\nabla\psi(\alpha(t)),\alpha'(t)\rangle=-\langle\text{Ric}(\hat\xi)\xi,X\rangle_{\alpha(t)}=0.$$
Hence, $\psi$ is constant along $\alpha$ and, since $\psi(q)=0$, we get that $\psi\equiv 0$ along $\alpha$. Since $X$ was arbitrarily chosen subjected to the condition $X\bot\xi,J\xi$, we conclude that $\psi\equiv 0$ in a neighborhood of $p$ in $M\setminus\xi^{-1}(0)$. 

Now, the discussion on Section $1$ of \cite{Obata:70} assures that $\psi$ and $\xi$ are uniquely determined by the values of $\psi$, $\nabla\psi$, $\xi$ and $\nabla\xi$ at a single point of $M$. Therefore, since $\psi$ vanishes on an open subset of $M\setminus\xi^{-1}(0)$, and (as we have observed above, for $n>1$) such a set is connected, we conclude that $\psi\equiv 0$ on $M\setminus\xi^{-1}(0)$. However, since $\xi^{-1}(0)$ is a set of isolated points, we actually have that $\psi\equiv 0$ on $M$. In turn, this shows that both $\xi$ and $J\xi$ (since $\nabla J=0$) are parallel on $M$.

Finally, for $X\in\mathfrak X(M)$ we have $X\langle\xi,\xi\rangle=2\langle\nabla_X\xi,\xi\rangle=2\langle\psi X,\xi\rangle=0$, so that $|\xi|^2$ is constant on $M$. Since $\xi=0$ at most at a set of isolated points, this implies that $|J\xi|=|\xi|$ is constant and positive on $M$.
\end{proof}
 
We shall need one more auxiliary result. 

\begin{lemma}\label{lemma:on parallel vector fields}
Let $M$ be a riemannian manifold, $\eta\in\mathfrak X(M)\setminus\{0\}$ be a parallel and complete vector field with flow $\Phi:\mathbb R\times M\rightarrow M$ and $\alpha:[0,1]\rightarrow M$ be a geodesic preserved by $\Phi$. If $P_{\alpha}:T_{\alpha(0)}M\rightarrow T_{\alpha(1)}M$ stands for the parallel transport along $\alpha$, then $P_{\alpha}=(d\Phi^1)_{\alpha(0)}$.
\end{lemma}

\begin{proof}
Let $p=\alpha(0)$ and $v\in T_pM$. It suffices to show that $t\mapsto(d\Phi^t)_p(v)$ is parallel along $\alpha$. To this end, let $\delta:(-\epsilon,\epsilon)\rightarrow M$ be such that $\delta(0)=p$ and $\delta'(0)=v$. Then, 
\begin{equation}\nonumber
 \begin{split}
\frac{D}{dt}(d\Phi^t)_p(v)&\,=\frac{D}{dt}\frac{\partial}{\partial s}\Phi(t,\delta(s))\Big|_{s=0}=\frac{D}{ds}\frac{\partial}{\partial t}\Phi(t,\delta(s))\Big|_{s=0}\\
&\,=\frac{D}{ds}\eta\big(\Phi(t,\delta(s))\big)\Big|_{s=0}=\Big(\nabla_{\frac{d}{ds}\Phi(t,\delta(s))}\eta\Big)_{\Phi(t,p)}=0,
 \end{split}
\end{equation}
since $\eta$ is parallel.
\end{proof}

For what follows, we recall that, given a connected submanifold $N$ of a riemannian manifold $N$, $p\in N$ and a closed and piecewise differentiable curve $\alpha:[0,1]\rightarrow N$ such that $\alpha(0)=p$, parallel translation along $\alpha$ defines a linear operator $P_{\alpha}:T_pN^{\bot}\rightarrow T_pN^{\bot}$. More precisely, for $v\in T_pN^{\bot}$, we set $P_{\alpha}(v)=V(1)$, where $t\mapsto V(t)$ is the parallel transport of $v$ along $\alpha$. It is immediate to check that the set of such linear operators, furnished with the product $P_{\alpha}\cdot P_{\beta}=P_{\alpha\cdot\beta}$, form a closed subgroup of $O(T_pN^{\bot})$, called the normal holonomy subgroup of $N$ at $p$ and denoted by $\text{Hol}_p^{\bot}(N)$ (for more details, see Chapter $4$ of \cite{Berndt:03}). If $q\in N$ and $\delta$ is a piecewise smooth curve in $N$ joining $p$ and $q$, we have $\text{Hol}_p^{\bot}(N)\simeq\text{Hol}_q^{\bot}(N)$ via $P_{\alpha}\mapsto P_{\delta^{-1}\cdot\alpha\cdot\delta}$. Therefore, from now on we shall refer to the normal holonomy group of $N$, which will henceforth be denoted by $\text{Hol}^{\bot}(N)$.

We have finally arrived at our second main result.

\begin{theorem}\label{thm:main theorem}
Let $n>1$ be an integer and $(M^n,g,J,\nabla)$ be a connected, complete k\"ahlerian manifold furnished with a closed, conformal and nontrivial vector field $\xi$. Let $\mathcal D$ denote the distribution in $M\setminus\xi^{-1}(0)$ generated by $\xi$ and $J\xi$, and assume that $\mathcal D$ has a compact leaf $\Sigma$. If the holomorphic sectional curvature of $M$ along $\Sigma$ is nonpositive and $\text{\rm Hol}^{\bot}(\Sigma)$ is a torsion group, then:
\begin{enumerate}
 \item[(a)] $\xi^{-1}(0)=\emptyset$ and both $\xi$ and $J\xi$ are parallel along $M$.
 \item[(b)] The leaves of $\mathcal D$, furnished with the induced metric, form a family of totally geodesic isometric tori.
 \item[(c)] The distribution $\mathcal D^{\bot}$ is integrable and, in the induced metric, their leaves are  totally geodesic, isometric complete k\"ah\-ler\-ian manifolds of complex dimension $n-1$.
 \end{enumerate}
\end{theorem}

\begin{proof}
Lemma \ref{lemma:folhas sao totalmente geodesicas} assures that $\Sigma$ is totally geodesic in $M$, so that its gaussian curvature $K_{\Sigma}$ coincides with the holomorphic sectional curvature of $M$ along $\Sigma$. Therefore, $K_{\Sigma}\leq 0$, and we can apply Theorem \ref{thm:the compact case} to conclude that $\Sigma$ is isometric to a flat torus (hence, $K_{\Sigma}\equiv 0$) and the conformal factor $\psi$ of $\xi$ vanishes along it. Lemma \ref{lemma:anulamento do fator conforme} then shows that $\psi\equiv 0$ on $M$, $\xi^{-1}(0)=\emptyset$ and $\xi$ and $J\xi$ are parallel along $M$. This establishes (a).

For (b), let $d_M$ stand for the riemannian distance on $M$. Given $q\in M$, the compactness of $\Sigma$ assures the existence of a point $p\in\Sigma$ such that $d_M(p,q)=d_M(p;\Sigma)$. Since $M$ is complete, Hopf-Rinow theorem guarantees the existence of a normalized geodesic $\gamma:[0,l]\rightarrow M$ such that $\gamma(0)=p$, $\gamma(l)=q$ and $l=\ell(\gamma)=d_M(p;q)$. In particular, it is a well known fact that $\gamma'(0)\bot T_p\Sigma$, so that $\langle\gamma'(0),\xi_p\rangle=\langle\gamma'(0),J\xi_p\rangle=0$. Now, the parallelism of $\xi$ assures that, along $\gamma$,
$$\frac{d}{ds}\langle\gamma',\xi\rangle_{\gamma(s)}=\langle\frac{D\gamma'}{ds},\xi\rangle_{\gamma(s)}+\langle\gamma',\frac{D\xi}{ds}\rangle_{\gamma(s)}=0.$$
Then, $\langle\gamma',\xi\rangle$ is constant along $\gamma$, so that $\langle\gamma',\xi\rangle_q=\langle\gamma',\xi\rangle_p=0$. Analogously, $\langle\gamma',J\xi\rangle_q=0$. On the other hand, if $\Sigma_q$ is the leaf of $\mathcal D$ passing through $q$, then $T_q\Sigma_q$ is generated by $\xi_q$ and $J\xi_q$, so that $\gamma'(l)\bot T_q\Sigma_q$.

By Lemma \ref{lemma:folhas sao totalmente geodesicas}, $\Sigma_q$ is totally geodesic in $M$. For $v\in T_q\Sigma_q$, the maximal geodesic of $\Sigma_q$ departing from $q$ with velocity $v$ coincides to that of $M$, which is complete. Hence, $\Sigma_q$ is also complete. Let $K_{\Sigma_q}$ stand for the gaussian curvature of $\Sigma_q$ and $K(\xi,J\xi)$ for the holomorphic sectional curvature of $M$ along the planes generated by $\xi$ and $J\xi$.  Letting $R$ denote the curvature operator of $M$, the parallelisms of $\xi$ and $J\xi$ give $R(J\xi,\xi)\xi=0$, so that $K(\xi,J\xi)\equiv 0$. However, since $\Sigma_q$ is totally geodesic in $M$, we conclude that $K_{\Sigma_q}=K(\xi,J\xi)_{|\Sigma_q}=0$. Therefore, being a connected, complete flat surface, $\Sigma_q$ is isometric to a torus, a plane or a cylinder over a plane curve. In what comes next, we shall show that is is isometric to a torus.

The proof of Theorem \ref{thm:the compact case} assures the existence of geodesic loops $\alpha_1$ and $\alpha_2$ in $\Sigma$, based at $p$ and such that $\pi_1(\Sigma;p)$ is generated by $[\alpha_1]$ and $[\alpha_2]$. Let $\alpha=\alpha_1:[0,1]\rightarrow\Sigma$ and write $\alpha'(0)=v=a\xi_p+bJ\xi_p$, for some $a,b\in\mathbb R$. If $v=a\xi_q+bJ\xi_q$, we already know that $v\in T_q\Sigma_q$. If $\beta$ is the geodesic of $\Sigma_q$ departing from $q$ with velocity $v$, we claim that $\beta$ is a geodesic loop, with lenght $\ell(\beta)=\ell(\alpha)$. We shall prove this in two steps:\\

\noindent (i) Let $w=\gamma'(0)\in T_p^{\bot}\Sigma$ and $\eta=a\xi+bJ\xi$, which is a parallel and (by the completeness of $M$) complete vector field in $M$. If $\Phi:\mathbb R\times M\rightarrow M$ is the flow of $\eta$, then $\Phi^t:M\rightarrow M$ is an isometry, so that $\gamma^t(s):=\Phi(t,\gamma(s))$ is a geodesic of $M$ joining $\alpha(t)$ to $\beta(t)$. The parallelism of $\eta$ assures, through Lemma \ref{lemma:on parallel vector fields}, that $w(t)=(d\Phi^t)_pw$ is parallel along $\alpha$ and normal to $\Sigma$. Hence, $w(1)=P_{\alpha}(w)$, with $P_{\alpha}\in\text{Hol}^{\bot}(\Sigma)$. Since this is a torsion group by hypothesis, there exists an integer $m\geq 1$ such that $P_{\alpha}^m=Id:T_p^{\bot}\Sigma\rightarrow T_p^{\bot}\Sigma$; in particular, $P_{\alpha}^m(w)=w$. However, since $P_{\alpha}^m=P_{\alpha\cdot\ldots\cdot\alpha}$ ($m$ times), if we set $c=\ell(\alpha\cdot\ldots\cdot\alpha)=m\ell(\alpha)$ we conclude that $\gamma^c:[0,l]\rightarrow M$ is a geodesic joining $p=(\alpha\cdot\ldots\cdot\alpha)(0)=p$ to $\beta(c)$ and with initial velocity $w$. Then, by uniqueness we get $\gamma^c=\gamma$, so that $\beta(c)=\gamma^c(l)=\gamma(l)=q$. Since $\beta'(0)$ is parallel to $\alpha'(0)$ along $\gamma$, $\beta'(c)$ is parallel to $\alpha'(c)$ along $\gamma^c$ and $\alpha'(0)=\alpha'(c)$, $\gamma=\gamma^c$, we conclude that $\beta'(0)=\beta'(c)$.\\

\noindent (ii) For $0\leq t\leq l$, let $v_t\in\Sigma_{\gamma(t)}$ denote the parallel transport of $v$ along $\gamma_{|[0,t]}$ and (and according to the discussion in (i)) $\delta_t$ be the geodesic loop that departs from $\gamma(t)$ with initial velocity $v_t$, so that $\delta_0=\alpha$ and $\delta_l=\beta$. Since $\exp:TM\rightarrow M$ is smooth, the function $t\mapsto\ell(\delta_t)$ is continuous; on the other hand (also from (i)), $\frac{\ell(\delta_t)}{\ell(\alpha)}\in\mathbb Z$ for each $t\in[0,l]$. In particular, $\ell(\beta)=\ell(\delta_l)=\ell(\delta_0)=\ell(\alpha)$.\\

The argument in (i) guarantees that the geodesics $\beta_1$ and $\beta_2$ of $\Sigma_q$, obtained from $\alpha_1$ and $\alpha_2$ by parallel transporting $\alpha_1'(0)$ and $\alpha_2'(0)$ along $\gamma$, are distinct geodesic loops in $\Sigma_q$. Since a plane has no closed geodesics and a cylinder over a plane curve has only one geodesic loop through each one of its points, we conclude that $\Sigma_q$ is a torus. Also, it follows by construction that $\langle\alpha_1'(0),\alpha_2'(0)\rangle=\langle\beta_1'(0),\beta_2'(0)\rangle$, and (ii) assures that $\ell(\alpha_1)=\ell(\beta_1)$ and $\ell(\alpha_2)=\ell(\beta_2)$. Therefore, the last part of the proof of Theorem \ref{thm:the compact case} assures that $\Sigma_q$ is isometric to $\Sigma$.

We now turn to (c). If $X$ and $Y$ are smooth vector fields in $\mathcal D^{\bot}$, then the parallelisms of $\xi$ and $J\xi$ give
$$\langle\nabla_XY,\xi\rangle=X\langle Y,\xi\rangle-\langle Y,\nabla_X\xi\rangle=0,$$
and accordingly $\langle\nabla_XY,J\xi\rangle=0$. In particular, $[X,Y]\in\mathcal D^{\bot}$, so that $\mathcal D^{\bot}$ is integrable. 

Letting $N$ denote a leaf of $\mathcal D^{\bot}$ and $\alpha$ its second fundamental form, we have
$$\alpha(X,Y)=\langle\nabla_XY,\xi\rangle\frac{\xi}{|\xi|^2}+\langle\nabla_XY,J\xi\rangle\frac{J\xi}{|\xi|^2}=0,$$
and $N$ is totally geodesic in $M$. The completeness of $N$ (in the induced metric) now follows from that of $M$, together with the fact that geodesics in $N$ are also geodesics in $M$.

It is immediate to check that $X\in\mathfrak X(N)\Rightarrow JX\in\mathfrak X(N)$. Therefore, $J$ is an almost complex structure on $N$, and the fact that the Levi-Civita connection of $N$ is the restriction of that of $M$ guarantees that $J$ is parallel on $N$. Finally, since the Nijenhuis tensor of $N$ is the restriction of that of $M$, which vanishes identically, we conclude that $N$ is a k\"ahlerian manifold in the induced metric. 

For the last part, we argue pretty much as in (i). To this end, let $N_1$ and $N_2$ be two distinct leaves of $\mathcal D^{\bot}$, and take $p_1\in N_1\cap\Sigma$ and $p_2\in N_2\cap\Sigma$. Let $\delta:[0,a]\rightarrow\Sigma$ be a geodesic of $\Sigma$ joining $p_1$ to $p_2$, and $\delta'(0)=a\xi_{p_1}+bJ\xi_{p_1}$ for some $a,b\in\mathbb R$. The parallelisms of $\xi$ and $J\xi$ assure that $\delta'$ is the restriction of the parallel (hence, complete) vector field $\eta=a\xi+bJ\xi$ to $\delta$.  If $\Phi:\mathbb R\times M\rightarrow M$  denotes the flow of $\eta$, then $\Phi^a:M\rightarrow M$ is an isometry such that $\Phi^a(p_1)=p_2$ and $(d\Phi^a)_{p_1}(T_{p_1}N_1)=T_{p_2}N_2$. Since $N_1$ and $N_2$ are connected, complete and totally geodesic in $M$, an argument pretty much like the one presented in the proof of (b) guarantees that $\Phi^a$ applies geodesics in $N_1$ to geodesics in $N_2$. Hence, $\Phi^a(N_1)\subset N_2$ and, likewise, $\Phi^{-a}(N_2)\subset N_1$. Thus, $\Phi^a(N_1)=N_2$.
\end{proof}

\begin{corollary}\label{coro:of the main theorem}
Let $n>1$ be an integer and $(M^n,g,J,\nabla)$ be a connected, complete k\"ahlerian manifold satisfying the hypotheses of the previous result. If $\tilde M$ stands for the universal covering of $M$, furnished with the covering metric, then $\tilde M$ is isometric to a riemannian product $\tilde N\times\mathbb R^2$, where $\tilde N$ is a connected, simply connected, complete k\"ahlerian manifold.
\end{corollary}

\begin{proof}
Letting $\tilde g$ denote the covering metric, the covering map $\pi:\tilde M\rightarrow M$ turns into a local isometry, so that $\tilde M$ is naturally a k\"ahlerian manifold. Moreover, the orthogonal foliations on $M$ lift to two orthogonal foliations of $\tilde M$ with totally geodesic leaves, and one of which has leaves isometric to $\mathbb R^2$. Since $\tilde M$ is simply connected, it now suffices to apply the complex version of the De Rham decomposition theorem.
\end{proof}

We finish the paper by presenting a generic example that, in a certain sense, sheds light into the assumption of nonpositivity of the holomorphic sectional curvature of $M$ along $\Sigma$, in Theorem \ref{thm:main theorem}.

\begin{example}\label{exam:kahlerianas sem campo conforme fechado}
Let $(M^n,g,J)$ be a compact complex symmetric space of positive scalar curvature ($n$ being its complex dimension). If $H^1_{dR}(M)=\{0\}$ and $M$ is not isometric to $\mathbb S^{2n}$, then $M$ does not possess a nontrivial closed conformal vector field. Indeed, if $\xi\in\mathfrak X(M)$ was such a vector field, then, since its metrically dual $1-$form $\theta_\xi$ is closed, it would generate an element $[\theta_{\xi}]\in H^1_{dR}(M)$, so that $[\theta_{\xi}]=0$. In turn, this would imply that $\xi=\nabla f$, for some smooth function $f$ on $M$. Then, the compactness of $M$ would give at least two zeros for $\xi$ in $M$. However, since $M$ has constant (and, by assumption, positive) scalar curvature, Theorem $1$ of \cite{Tanno:69} would imply that $M$ should be isometric to the euclidean sphere $\mathbb S^{2n}$, which is a contradiction.

A particular instance of such a situation is that of $\mathbb C\mathbb P^n$ furnished with the Fubini-Study metric. Actually, if $(M^n,g,J)$ is a compact connected k\"ahlerian manifold of positive constant holomorphic sectional curvature ($n$ being its complex dimension), then $M$ does not posses a nontrivial closed conformal vector field too. Indeed, the universal covering of $M^n$, furnished with the covering metric is isometric to $\mathbb C\mathbb P^n$ furnished with the Fubini-Study metric; if there existed a nontrivial closed conformal vector field in $M$, this could be lifted to a nontrivial closed conformal vector field in $\mathbb C\mathbb P^n$, which is impossible.
\end{example}

\end{document}